\documentclass[oneside,english]{amsart}
\usepackage[T1]{fontenc}
\usepackage[latin9]{inputenc}
\usepackage{amsbsy}
\usepackage{amstext}
\usepackage{amsthm}
\usepackage{amssymb}
\usepackage[numbers]{natbib}
\usepackage[all]{xy}

\makeatletter
\numberwithin{equation}{section}
\numberwithin{figure}{section}
\theoremstyle{plain}
\newtheorem{thm}{\protect\theoremname}
\theoremstyle{plain}
\newtheorem{lem}[thm]{\protect\lemmaname}
\theoremstyle{plain}
\newtheorem{cor}[thm]{\protect\corollaryname}
\theoremstyle{definition}
\newtheorem{defn}[thm]{\protect\definitionname}
\theoremstyle{definition}
\newtheorem{example}[thm]{\protect\examplename}

\usepackage{amsmath,amssymb,url}

\numberwithin{equation}{section}

\theoremstyle{plain}

\theoremstyle{definition}

\makeatletter
\newcommand{\xyR}[1]{%
\makeatletter
\xydef@\xymatrixrowsep@{#1}
\makeatother
} 

\makeatletter
\newcommand{\xyC}[1]{%
\makeatletter
\xydef@\xymatrixcolsep@{#1}
\makeatother
} 

\makeatother

\usepackage{babel}
\providecommand{\corollaryname}{Corollary}
\providecommand{\definitionname}{Definition}
\providecommand{\examplename}{Example}
\providecommand{\lemmaname}{Lemma}
\providecommand{\theoremname}{Theorem}

\begin{document}
\title{Connected monads weakly preserve products}
\author{H.Peter Gumm}
\address{Philipps-Universität Marburg, Marburg, Germany.}
\email{gumm@mathematik.uni-marburg.de}
\thanks{I am sincerely indebted to Peter Jipsen and Andrew Moshier for inspiring
discussions during my stay at Chapman University, where the main result
of this note was obtained.}
\begin{abstract}
If $F$ is a (not necessarily associative) monad on $Set$, then the
natural transformation $F(A\times B)\to F(A)\times F(B)$ is surjective
if and only if $F(\boldsymbol{1})=\boldsymbol{1}$. Specializing $F$
to $F_{\mathcal{V}}$, the free algebra functor for a variety $\mathcal{V},$
this result generalizes and clarifies an observation by Dent, Kearnes
and Szendrei in \citep{modularityTest}.
\end{abstract}

\maketitle

\section{Introduction}

A key observation in \citep{modularityTest} by T.~Dent, K.~Kearnes,
and Á.~Szendrei is that for any variety $\mathcal{V}$ with idempotent
operations each set theoretic product decomposition 
\[
d:\{x,y,z,u\}\twoheadrightarrow\{a,b\}\times\{a,b\}
\]
always extends to a surjective homomorphism 
\begin{equation}
\delta:F_{\mathcal{V}}(\{x,y,z,u\})\twoheadrightarrow F_{\mathcal{V}}(\{a,b\})\times F_{\mathcal{V}}(\{a,b\})\label{eq:surjective}
\end{equation}
 from the 4-generated free algebra in $\mathcal{V}$ to the square
of the 2-generated one.

This fact has an interesting geometric interpretation, which is relevant
in the study of congruence modularity. The shifting lemma from \citep{Gum83},
which is concerned with shifting a congruence $\gamma$ from one side
of an $\alpha-\beta$-parallelogram to the opposite side modulo $\alpha\wedge\beta$,
can be specialized to axis-parallel rectangles inside a product of
algebras where $\alpha$ and $\beta$ are in fact kernels of the projections
and $\gamma$ a factor congruence.

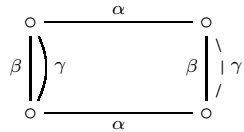
\[
\xymatrix{\circ\ar@{-}[d]_{\beta}\ar@{-}[rr]^{\alpha}\ar@{-}@/^{0.5pc}/[d]^{\gamma} & \, & \circ\ar@{-}[d]_{\beta}\ar@{--}@/^{0.5pc}/[d]^{\gamma}\\
\circ\ar@{-}[rr]_{\alpha} & \, & \circ
}
\]

Surjectivity of the above map implies that the projections on the
image commute, and since $ker\,\delta=\alpha\wedge\beta$, it follows
that $\alpha$ and $\beta$ also commute in the preimage. In particular,
therefore, the shifting lemma, which in \citep{Gum83} is the major
geometrical tool for studying congruence modularity, is only needed
in situations of permuting congruence relations $\alpha$ and $\beta$.
The restriction to idempotent varieties in these studies is not severe,
since a variety is congruence modular iff its idempotent reduct is
modular.

Variations of the shifting lemma (e.g. in \citep{Chajdan2003}) and,
more recently, categorical generalizations as in \citep{Gran2019}
suggest to investigate the situation in a more general context. In
this note, therefore, rather than exploring further ramifications
of the above observation, we explore the abstract reasons behind the
surjectivity of $\delta$ in (\ref{eq:surjective}). It turns out
that we can deal with this in a framework which is more abstract than
universal algebras and varieties. We are rather considering (not necessarily
associative) $Set$-monads $F$, of which the functor $F_{\mathcal{V}}$,
associating with a set $X$ the free algebra $F_{\mathcal{V}}(X)$
and with a map $g:X\to Y$ its homomorphic extension $\bar{g}:F_{\mathcal{V}}(X)\to F_{\mathcal{V}}(Y)$,
is just an example.

\section{Monads and main result}

Monads on a category $\mathcal{C}$ are functors $F:\mathcal{C}\to\mathcal{C}$
together with two natural transformations $\iota:Id\to F$ and $\mu:F\circ F\to F$,
satisfying two unit laws and and an associative law. Our results will
even hold for nonassociative monads, so skippping the associative
law, we shall only state the unit laws: 
\begin{equation}
\mu_{X}\circ\iota_{F(X)}=id_{F(X)}=\mu_{X}\circ F\iota_{X}\label{eq:unit laws}
\end{equation}
Equations (\ref{eq:unit laws}) are usually expressed as a commutative
diagram: 
\[
\xymatrix{F(X)\ar[r]^{\iota_{F(X)}} & F(F(X))\ar[d]^{\mu_{X}} & F(X)\ar[l]_{F\iota_{X}}\\
 & F(X)\ar@{=}[ul]\ar@{=}[ur]
}
\]

Rather easy examples of monads on the category of Sets are obtained
from collection data types in programming, such as $List\langle X\rangle$,
$Set\langle X\rangle$ or $Tree\langle X\rangle$, see also \citep{Manes98}.
In popular programming languages, $List\langle X\rangle$ denotes
the type of lists of elements from a base type $X.$ Given a function
$g:X\to Y$, the function $map(g):List\langle X\rangle\to List\langle Y\rangle$
which sends $[x_{1},...,x_{n}]\in List\langle X\rangle$ to the list
$[g(x_{1}),...,g(x_{n})]\in List\langle Y\rangle$ represents the
action of the functor $List$ on maps. In mathematical notation we
write $(List~g)$ rather than $map(g)$. Obviously, $map(f\circ g)=map(f)\circ map(g)$
and $map(id_{X})=id_{List\langle X\rangle},$ so the pair $List\langle-\rangle$
with $map$ indeed establishes a functor.

For $List$ to be a monad, we need a natural transformation $\iota:Id\to List$,
as well as a ``multiplication'' $\mu:List\circ List\to List$. The
former can be chosen as the \emph{singleton} operator with $\iota_{X}:X\to List\langle X\rangle$
sending any $x\in X$ to the one-element list $[x]$.

The monad multiplication $\mu$ is for each type $X$ defined as 
\[
\mu_{X}:List\langle List\langle X\rangle\rangle\to List\langle X\rangle,
\]
taking a list of lists $[l_{1},...,l_{n}]$ and appending them into
a single list $l_{1}+...+l_{n}$. Programmers call this operation
\emph{``flatten}''. The unit laws then state that for each list
$l=[x_{1},...,x_{n}]\in List\langle X\rangle$ we should have 
\[
flatten([\,[x_{1},..,x_{n}]\,])=[x_{1},...,x_{n}]=flatten([\,[x_{1}],...,[x_{n}]\,],
\]
which is obvious. Not all monads arise from collection classes, and
in recent years other uses of monads have all but revolutionized functional
programming, see e.g. \citep{MOGGI199155} or \citep{Wad92}.

Relevant for universal algebraists is the fact that for every variety
$\mathcal{V}$ the construction of the free Algebra $F_{\mathcal{V}}(X)$
over a set $X$ is a monadic functor. In this case, $\iota_{X}:X\to F_{\mathcal{V}}(X)$
is the inclusion of variables, or rather their interpretations as
$\mathcal{V}-$terms.

The defining property of $F_{\mathcal{V}}(X)$ states that each map
$g:X\to A$ for $A\in\mathcal{V}$ has a unique homomorphic extension
$\bar{g}:F_{\mathcal{V}}(X)\to A$.

From a map $f:X\to Y$, we therefore obtain the homomorphism 
\[
F_{\mathcal{V}}f:F_{\mathcal{V}}(X)\to F_{\mathcal{V}}(Y)
\]
as unique homomorphic extension of the composition $\iota_{Y}\circ f:X\to F_{V}(Y)$:

\[
\xymatrix{F_{\mathcal{V}}(X)\ar[r]^{F_{\mathcal{V}}f} & F_{\mathcal{V}}(Y)\\
X\ar@{^{(}->}[u]^{\iota_{X}}\ar[r]^{f}\ar@{..>}[ur] & Y\ar@{^{(}->}[u]_{\iota_{Y}}
}
\]

$\mu:F_{\mathcal{V}}(F_{\mathcal{V}}(X))\to F_{\mathcal{V}}(X)$ can
be considered as term composition: a term $t(t_{1},...,t_{n}),$ whose
argument positions have been filled by other terms, is interpreted
as an honest $\mathcal{V}$-term. To make this precise, consider the
diagram below, in which $F_{\mathcal{V}}(X)$ appears in two roles
-- as an algebra and as a set of free variables for $F_{\mathcal{V}}(F_{\mathcal{V}}(X))$.
(In the diagram we have dropped the lower indices to $\iota$ and
$id$ for the sake of readibility.)

Here $\mu_{X}$ is defined as the homomorphic extension of the equality
map $id_{F_{\mathcal{V}}(X)}$ from $F_{\mathcal{V}}(X)$, considered
as set of free variables for $F_{\mathcal{V}}(F_{\mathcal{V}}(X))$,
to $F_{\mathcal{V}}(X)$ considered as $\mathcal{V}$-algebra. 
\[
\xymatrix{ & F_{\mathcal{V}}(F_{\mathcal{V}}(X))\ar@<.2pc>[dl]^{\mu_{X}}\\
\xyC{3pc}F_{\mathcal{V}}(X)\ar[r]_{id}\ar@<.2pc>[ur]^{F_{\mathcal{V}}\iota} & F_{\mathcal{V}}(X)\ar[l]\ar@{^{(}->}[u]_{\iota}\\
X\ar@{^{(}->}[u]^{\iota}\ar@{^{(}->}[ur]_{\iota}
}
\]
The first monad equation immediately follows from the definition of
$\mu,$ and the second equation

\[
\mu\circ(F_{\mathcal{V}}\iota_{X})=id_{F_{\mathcal{V}}(X)}
\]
follows from the calculation
\begin{eqnarray*}
\mu\circ(F_{\mathcal{V}}\iota_{X})\circ\iota_{X} & = & \mu\circ\iota_{F_{\mathcal{V}}(X)}\circ\iota_{X}\\
 & = & id_{F_{\mathcal{V}}(X)}\circ\iota_{X}
\end{eqnarray*}
demonstrating that both sides agree on the generators of $F_{\mathcal{V}}(X)$,
and consequently on all of $F_{\mathcal{V}}(X)$.

The above mentioned examples $Tree\langle X\rangle$, $List\langle X\rangle$
and $Set\langle X\rangle$ just correspond to the free groupoid, the
free semigroup, and the free semilattice over the set $X$ of generators,
and are themselves instances of this scheme.

We are now ready to state our main result:
\begin{thm}
\label{thm:main}A (not necessarily associative) Set-monad $F$ weakly
preserves products if and only if $F(\boldsymbol{1})\cong\boldsymbol{1}$.
\end{thm}

It will be easy to see (lemma \ref{lem:weak product preservation}
below) that \emph{$F$ weakly preserves the product $A_{1}\times A_{2}$
}if and only if the canonical morphism $\delta=(F\pi_{1},F\pi_{2})$
in the below diagram is epi:

\begin{equation}
\xymatrix{F(A_{1}\times A_{2})\ar[dr]_{F\pi_{i}}\ar[rr]^{\delta} &  & F(A_{1})\times F(A_{2})\ar[dl]^{\eta_{i}}\\
 & F(A_{i})
}
\label{eq:delta}
\end{equation}
The starting point of our discussion, (\ref{eq:surjective}) from
\citep{modularityTest}, is therefore seen to represent an instance
of this result when setting $A_{1}=A_{2}=\{a,b\}$ and $F=F_{\mathcal{V}}$.
But before coming to its proof we need a few preparations.

\section{Connected Functors}

Put $\boldsymbol{1}=\{0\}$ and for any set $X$ denote by $!_{X}$
the unique (terminal) map from $X$ to \textbf{$\boldsymbol{1}$}.
A $Set$-functor $F$ is called \emph{connected}, if $F(\boldsymbol{1})\cong\boldsymbol{1}$.
Given a variety $\mathcal{V}$, the functor $F_{\mathcal{V}}$ is
connected if and only if $\mathcal{V}$ is idempotent.

It is well known, see \citep{TrnkovaDescr1}, that every $Set$-Functor
$F$ can be constructed as sum of connected functors: 
\[
F=\Sigma_{e\in F(1)}F_{e}.
\]
For $e\in F(\boldsymbol{1})$ one simply puts $F_{e}(X)=\{u\in F(X)\mid(F!_{X})(u)\}.$
On maps $f:X\to Y$, each subfunctor $F_{e}$ is just the domain-codomain-restriction
of $Ff$ to $F_{e}(X)$.

In the following we denote by $c_{y}^{X}:X\to Y$ or, if $X$ is clear,
simply by $c_{y}$ the constant map with value $y\in Y.$ We shall
need the following lemma:
\begin{lem}
\label{lem:F(constant)}If $F$ is a connected functor, then $Fc_{y}^{X}$
is a constant map. Whenever \textup{$\iota:Id\to F$} is a natural
transformation, then \textup{$Fc_{y}^{X}=c_{\iota_{Y}(y)}^{F(X)}$}.
\end{lem}

\begin{proof}
For $y\in Y,$ denote by $\bar{y}:\boldsymbol{1}\to Y$ the constant
map with value $y.$ Observe, that an arbitrary map $f$ is constant
if and only if it factors through $1$, i.e. $c_{y}^{X}=\bar{y}\circ!_{X}$.
Applying $F$ and adding the natural transformation $\iota$ into
the picture,

\[
\xyC{1pc}\xyR{1pc}\xymatrix{F(X)\ar[dr]_{F!_{X}}\ar[rr]^{Fc_{y}^{X}}\ar@/_{2.8pc}/[ddrr]_{!_{F(X)}} &  & F(Y)\\
 & F(\boldsymbol{1})\cong\boldsymbol{1}\ar[ur]_{F\bar{y}} &  & Y\ar[ul]_{\iota_{Y}}\\
 &  & \,\,\boldsymbol{1}\,\,\ar@{>->>}[ul]^{\iota_{\boldsymbol{1}}}\ar[ur]_{\bar{y}}
}
\]
 we obtain:
\begin{eqnarray*}
Fc_{y}^{X} & = & F\bar{y}\circ F!_{X}\\
 & = & F\bar{y}\circ\iota_{\boldsymbol{1}}\,\circ\,!_{F(X)}\\
 & = & \iota_{Y}\circ\bar{y}\,\circ\,!_{F(X)}\\
 & = & \overline{\iota_{Y}(y)}\,\circ\,!_{F(X)}\\
 & = & c_{\iota_{Y}(y)}^{F(X)}.
\end{eqnarray*}
\end{proof}
In the above, we have seen, that connected functors preserve constant
maps. It might be interesting to remark, that this very property characterizes
connected functors:
\begin{cor}
\label{cor:connected functors preserve constants}A functor $F$ is
connected if and only if for every constant morphism $c_{y}$ the
morphism $Fc_{y}$ is constant, again.
\end{cor}

\begin{proof}
Suppose that $F$ preserves constant maps. Since $id_{\boldsymbol{1}}$
is constant, $F(id_{\boldsymbol{1}})=id_{F(\boldsymbol{1})}$ must
be constant, too, which implies $F(\boldsymbol{1})\cong\boldsymbol{1}$.
\end{proof}
In general, the elements of $F(\boldsymbol{1})$ correspond uniquely
to the natural transformations between the identity functor $Id$
and $F$. This can be seen by instantiating the Yoneda Lemma 
\begin{equation}
nat(Hom(A,-),F(-))\cong F(A)\label{eq:yoneda}
\end{equation}
with $A=\boldsymbol{1}.$ Therefore we note:
\begin{cor}
A monad $(F,\iota,\mu)$ is connected if and only if $\iota$ is the
only transformation from the identity functor to $F$.
\end{cor}

\begin{defn}
Let $\mathcal{C}_{\boldsymbol{1}}$ be the constant functor with $\mathcal{C}_{\boldsymbol{1}}(X)=\boldsymbol{1}$
for all $X$ and $\mathcal{C}_{\boldsymbol{1}}f=id_{\boldsymbol{1}}$
for all $f$. We say that a functor $F$ \emph{possesses a constant},
if there is a transformation from $\mathcal{C}_{\boldsymbol{1}}$
to $F$ which is natural, except perhaps at $X=\emptyset$.
\end{defn}

Clearly, each element of $F(\emptyset)$ gives rise to a constant,
but not conversely, since there is nothing to stop us from changing
$F$ only on the empty set $\emptyset$ and on empty mappings $\emptyset_{X}:\emptyset\to X$
by choosing any $U\subseteq F(\emptyset)$ and redefining $F'(\emptyset):=U$
as well as $F'\emptyset_{X}=F\emptyset_{X}\circ\subseteq_{U}^{X}.$
For that reason we were not requiring naturality at $\emptyset$ in
the above definition.

We shall need a further observation:
\begin{lem}
\label{lem:constant subfunctor}A connected functor either possesses
a constant or it has the identity functor as a subfunctor.
\end{lem}

\begin{proof}
By the Yoneda-Lemma, there is exactly one natural transformation $\iota:Id\to F$.
Assume that some $\iota_{X}$ is not injective, then there are $x_{1}\ne x_{2}\in X$
with $\iota_{X}(x_{1})=\iota_{X}(x_{2})$. Given an an arbitrary $Y$
with $y_{1},y_{2}\in Y,$ consider a map $f:X\to Y$with $f(x_{1})=y_{1}$
and $f(x_{2})=y_{2}$. By naturality, 
\[
\iota_{Y}(y_{1})=\iota_{Y}(f(x_{1}))=Ff\circ\iota_{X}(x_{1})=Ff\circ\iota_{X}(x_{2})=\iota_{Y}(y_{2}),
\]
hence each $\iota_{Y}$ is konstant and therefore factors through
$\boldsymbol{1}$. This makes the upper and lower triangle inside
the following naturality square commute, too. The left triangle commutes
since $\boldsymbol{1}$ is terminal. If $X\ne\emptyset$, the terminal
map $!_{X}:X\to1$ is epi, from which we now conclude that the right
triangle commutes as well, except, possibly, when $X=\emptyset$.
Thus $F$ posseses a constant.

\[
\xymatrix{\,\ar@<.2pc>[r]^{\bar{x_{1}}}\ar@<-.2pc>[r]_{\bar{x_{2}}} & X\ar[dd]_{f}\ar[rr]^{\iota_{X}}\ar[dr] &  & F(X)\ar[dd]^{Ff}\\
 &  & \boldsymbol{\,\,\,\,1\,\,\,}\ar[ur]\ar[dr]\\
\,\ar@<.2pc>[r]^{\bar{y_{1}}}\ar@<-.2pc>[r]_{\bar{y_{2}}} & Y\ar[rr]_{\iota_{Y}}\ar[ur] &  & F(Y)
}
\]
\end{proof}

\section{Preservation properties}

We are concerned with the question, under which conditions the $\delta$
in equation (\ref{eq:surjective}) is epi. Therefore, we take a look
at the canonical map $\delta=(F\pi_{1},F\pi_{2}):F(A_{1}\times A_{2})\to F(A_{1})\times F(A_{2})$
which arises from the commutative diagram (\ref{eq:delta}), where
$\pi_{i},$ resp $\eta_{i},$ denote the canonical component projections.

The first thing to observe is:
\begin{lem}
$\delta=(F\pi_{1},F\pi_{2}):F(A_{1}\times A_{2})\to FA_{1}\times FA_{2}$
is natural in each component.
\end{lem}

\begin{proof}
Assume $f:A_{1}\to A_{1}'$ and $g:A_{2}\to A_{2}'$ be given. We
want to show that the following diagram commutes:
\[
\xyC{5pc}\xyR{2pc}\xymatrix{F(A_{1}\times A_{2})\ar[r]^{(F\pi_{1},F\pi_{2})}\ar[d]_{F(f\times g)} & F(A_{1})\times F(A_{2})\ar[d]^{Ff\times Fg}\\
F(A_{1}'\times A_{2}')\ar[r]^{(F\pi'_{1},F\pi'_{2})} & F(A_{1}')\times F(A_{2}')
}
\]
We calculate:
\begin{eqnarray*}
((Ff\times Fg)\circ(F\pi_{1},F\pi_{2}))(u) & = & (Ff\times Fg)((F\pi_{1})(u),(F\pi_{2})(u))\\
 & = & ((Ff\circ F\pi_{1})(u),(Fg\circ F\pi_{2})(u))\\
 & = & (F(f\circ\pi_{1})(u),F(g\circ\pi_{2})(u))\\
 & = & (F(\pi_{1}'\circ f\times g)(u),F(\pi_{2}'\circ f\times g)(u))\\
 & = & ((F\pi_{1}'\circ F(f\times g))(u),(F\pi_{2}'\circ F(f\times g))(u))\\
 & = & (F(\pi_{1}')(F(f\times g)(u)),F(\pi_{2}')(F(f\times g)(u)))\\
 & = & (F\pi_{1}',F\pi_{2}')(F(f\times g)(u))\\
 & = & ((F\pi_{1}',F\pi_{2}')\circ F(f\times g))(u).
\end{eqnarray*}
\end{proof}
Notice, that in order for $\delta$ to be surjective, the functor
$F$ must be connected or trivial:
\begin{lem}
\label{lem:F1=00003D1}If the canonical decomposition as in Theorem
1 is always epi, then either $F(1)\cong1$ or $F$ is the trivial
functor with constant value $\emptyset$.
\end{lem}

\begin{proof}
For the projections $\pi_{1},\pi_{2}:\boldsymbol{1}\times\boldsymbol{1}\to\boldsymbol{1}$
we have $\pi_{1}=\pi_{2},$ since $\boldsymbol{1}$ is a terminal
object, hence also $F\pi_{1}=F\pi_{2}$. Let $\eta_{1},\eta_{2}$
be the projections from the product $F(\boldsymbol{1})\times F(\boldsymbol{1})$
to its components. Then 
\[
\eta_{1}\circ(F\pi_{1},F\pi_{2})=F\pi_{1}=F\pi_{2}=\eta_{2}\circ(F\pi_{1},F\pi_{2}).
\]
By assumption, $\delta=(F\pi_{1},F\pi_{2})$ is epi, so $\eta_{1}=\eta_{2}$.
For arbitrary $a,b\in F(\boldsymbol{1})$ then $(a,b)\in F(\boldsymbol{1})\times F(\boldsymbol{1}),$
so 
\[
a=\eta_{1}(a,b)=\eta_{2}(a,b)=b.
\]
So $F(\boldsymbol{1})$ either has just one element, or $F(\boldsymbol{1})=\emptyset$.
In the latter case, for each set $X$ the map $!_{X}:X\to1$ should
yield a map $F!_{X}:F(X)\to F(\boldsymbol{1}),$ so $F(\boldsymbol{1})=\emptyset$
implies $F(X)=\emptyset.$
\end{proof}
Next, recall some elementary categorical notions.
\begin{defn}
Given objects $A_{1},A_{2}$ in a category $\mathcal{C}$, a \emph{product
of} $A_{1}$ \emph{and} $A_{2}$ is an object $P$ together with morphisms
$p_{i}:P\to A_{i}$, such that for any ``competitor'', i.e. for
any object $Q$ with morphisms $q_{i}:Q\to A_{i}$, there exists\emph{
a unique }morphism $d:Q\to P$, such that $q_{i}=p_{i}\circ\delta$
for $i=1,2.$ Products, if they exist, are unique up to isomorphism
and are commonly written $A_{1}\times A_{2}$.

Similarly, given morphisms $f_{1}:X_{1}\to Y$ and $f_{2}:X_{2}\to Y$
with common codomain $Y,$ their \emph{pullback} is defined to be
a pair of maps $p_{1}:P\to X_{1}$ and $p_{2}:P\to X_{2}$ with common
domain $P$ such that 
\[
f_{1}\circ p_{1}=f_{2}\circ p_{2}
\]
and for each ``competitor'', i.e. each object $Q$ with morphisms
$q_{1}:Q\to X_{1}$ and $q_{2}:Q\to X_{2}$ also satisfying $f_{1}\circ q_{1}=f_{2}\circ q_{2}$
there exists \emph{a unique }morphism $d:Q\to P$ so that $p_{i}\circ d=q_{i}$
for $i=1,2.$
\[
\xyR{1pc}\xyC{2pc}\xymatrix{Q\ar@/{}_{1pc}/[ddr]_{q_{2}}\ar@/{}^{1pc}/[drr]^{q_{1}}\ar@{.>}[dr]|-{d} &  &  &  & Q\ar@/{}_{1pc}/[ddr]_{q_{2}}\ar@/{}^{1pc}/[drr]^{q_{1}}\ar@{.>}[dr]|-{d}\\
 & P\ar[d]^{p_{2}}\ar[r]_{p_{1}} & A_{1} & \text{resp.} &  & P\ar[d]^{p_{2}}\ar[r]_{p_{1}} & A_{1}\ar[d]_{f_{1}}\\
 & A_{2} &  &  &  & A_{2}\ar[r]^{f_{2}} & B
}
\]
In both definitions, if we drop the uniqueness requirement, we obtain
the definition of \emph{weak product}, resp. \emph{weak pullback}.
\end{defn}

Notice that in case when there exists a terminal object $\boldsymbol{1}$,
the product of $A_{1}$ with $A_{2}$ is the same as the pullback
of the terminal morphisms $!_{A_{i}}:A_{i}\to\boldsymbol{1}$.

Weak products (weak pullbacks) arise from right invertible morphisms
into products (pullbacks):
\begin{lem}
\label{lem:weak pullback}If $(P,p_{1},p_{2})$ is a product (resp.
pullback), then $(W,w_{1},w_{2})$ is a weak product (resp. weak pullback)
if and only if there is a right invertible $w:W\to P$ such that $w_{i}=p_{i}\circ w$.
\end{lem}

\begin{proof}
If $w$ has a right inverse $e$, and $(Q,q_{1},q_{2})$ is a competitor
to $W,$ then it is also a competitor to $P,$ hence there is a morphism
$d:Q\to P$ with $q_{i}=p_{i}\circ d$. Then $e\circ d$ is the required
morphism to $W$. Indeed, 
\[
w_{i}\circ(e\circ d)=p_{i}\circ w\circ e\circ d=p_{i}\circ d=q_{i}.
\]

Conversely, assume that $(W,w_{1},w_{2})$ is a weak product, then
both $W$ and $P$ are competitors to each other, yielding both a
morphism $w:W\to P$ with $w_{i}=p_{i}\circ w$ and a morphism $e:P\to W$
with $p_{i}=w_{i}\circ e$.

Now $(P,p_{1},p_{2})$ is also acompetitor to itself, yet both $p_{i}\circ(w\circ e)=p_{i}$
and $p_{i}\circ id_{P}=p_{i}$ for $i=1,2$. By uniqueness it follows,
$w\circ e=id_{P},$ so $w$ is indeed right invertible. (The same
proof works for the case of weak pullbacks).
\end{proof}
\begin{defn}
Let $F:\mathcal{C}\to\mathcal{D}$ be a functor. We say that $F$
\emph{weakly preserves} products (pullbacks) if whenever $(P,p_{1},p_{2})$
is a product (pullback), then its image $(F(P),Fp_{1},Fp_{2})$ is
a weak product (weak pullback).
\end{defn}

It is well known, that a functor weakly preserves a limit $L$, if
and only it preserves weak limits, see e.g. \citep{Gum98}. By the
axiom of choice, surjective maps are right invertible, so regarding
(\ref{eq:surjective}) or its more general formulation (\ref{eq:delta}),
we now arrive at the following relevant observation:
\begin{lem}
\label{lem:weak product preservation}The canonical map $\delta$
in (\ref{eq:delta}) is epi if and only if $F$ weakly preserves the
product $(A_{1}\times A_{2},\pi_{1},\pi_{2})$.
\end{lem}

Whereas the above mentioned result of \citep{modularityTest}, in
which the monad $F$ is the free-algebra-functor $F_{V}$, served
a purely universal algebraic purpose, it also has an interesting coalgebraic
interpretation. It is well known, that coalgebraic properties of classes
of $F$-coalgebras are to a large degree determined by \emph{weak
pullback preservation} properties of the functor $F$, which serves
as a \emph{type} or \emph{signature} for a class $Coalg_{F}$ of coalgebras.
Prominent structure theoretic properties can be derived from the assumptions
that $F$ weakly preserves pullbacks of preimages, kernel pairs or
both, see e.g. \citep{Gum99b},\citep{Gum98},\citep{GS99a},\citep{GS01},\citep{GS05},
\citep{Rut96}. Here we add one more property to this list: preservation
of pullbacks of constant maps.
\begin{thm}
\label{thm:equivalence}Let $F$ be a nontrivial functor. Then the
following are equivalent:
\begin{enumerate}
\item $F$ has no constant and weakly preserves products
\item $F$ is connected and weakly preserves pullbacks of constant maps.
\end{enumerate}
\end{thm}

\begin{proof}
If $F$ is nontrivial and weakly preserves the product $\boldsymbol{1}\times\boldsymbol{1}\cong\boldsymbol{1}$,
then $F$ is connected as a consequence of lemma \ref{lem:F1=00003D1}.
Since $F$ has no constants, $F(\emptyset)=\emptyset$ and moreover
lemma \ref{lem:constant subfunctor} provides $Id$ as a subfunctor
of $F$. Thus we obtain a natural transformation $\iota:Id\to F$
which is injective in each component.

Let now $c_{y_{i}}^{X_{i}}:X_{i}\to Y$ for $i=1,2$ be constant maps
with $y_{i}\in Y$. Applying $F$, lemma \ref{lem:F(constant)} yields
$Fc_{y_{i}}^{X_{i}}=c_{\iota_{Y}(y_{i})}^{F(X_{i})}$ for $i=1,2.$

If $y_{1}=y_{2}$ then the pullback of the $c_{y_{i}}^{X_{i}}$ is
simply $(X_{1}\times X_{2},\pi_{1},\pi_{2})$. The $Fc_{y_{i}}^{X_{i}}$
are constant maps with the same target value $\iota_{Y}(y_{1})=\iota_{Y}(y_{2}),$
so their pullback is the product $F(X_{1})\times F(X_{2})$ with canonical
projections $\eta_{i}:F(X_{1})\times F(X_{2})\to F(X_{i})$. By assumption,
$F$ weakly preserves products, which gives us a surjective canonical
map $\delta:F(X_{1}\times X_{2})\to F(X_{1})\times F(X_{2})$ with
$F\pi_{i}=\eta_{i}\circ\delta$, so lemma \ref{lem:weak pullback}
assures that $(F(X_{1}\times X_{2}),F\pi_{1},F\pi_{2})$ is a weak
pullback of the $Fc_{y_{i}}^{X_{i}}.$

If $y_{1}\ne y_{2}$ then the pullback of the $c_{y_{i}}^{X_{i}}$
is $(\emptyset,\emptyset_{X_{1}},\emptyset_{X_{2}})$, the empty set
$\emptyset$ with empty mappings $\emptyset_{X_{i}}:\emptyset\to X_{i}$.
Since $\iota_{Y}$ is injective, the $Fc_{y_{1}}$ are constant mappings,
also with disjoint images, so their pullback is $(\emptyset,\emptyset_{F(X_{1})},\emptyset_{F(X_{2})}).$
This is the same we would obtain by applying $F$ to the pullback
of the $c_{y_{i}}$, taking into account that $F(\emptyset)=\emptyset$.

For the reverse direction, suppose that $F$ is connected and weakly
preserves pullbacks of constant maps. The product $(X_{1}\times X_{2},\pi_{1},\pi_{2})$
is at the same time the pullback of the terminal maps $!_{X_{i}}:X_{i}\to\boldsymbol{1}$.
Applying $F$ and considering that $F(\boldsymbol{1})\cong\boldsymbol{1}$,
we see that the $F!_{X_{i}}$ are also terminal maps, so their pullback
is $(F(X_{1})\times F(X_{2}),\eta_{1},\eta_{2})$. Thus, if $F$ weakly
preserves the pullback of the $!_{X_{i}}$ we must have that $(F(X_{1}\times X_{2}),F\pi_{1},F\pi_{2})$
is a weak pullback of the $F!_{X_{i}}$ which by lemma \ref{lem:weak pullback}
means that there exists a surjective map $\delta:F(X_{1}\times X_{2})\to F(X_{1})\times F(X_{2})$
with $\eta_{i}\circ\delta=F\pi_{i}$.
\end{proof}
The following example shows that the requirement that $"F$ has no
constants'' is essential in theorem \ref{thm:equivalence}.
\begin{example}
Consider the functor $T$ with $T(X)=X^{2}/\Delta$ where $\Delta$
is the equivalence relation on $X^{2}$ identifying any two elements
in the diagonal of $X^{2}.$ For $x_{1},x_{2}\in X$, we denote the
elements of $X^{2}/\Delta$ by $(x_{1},x_{2})$ if $x_{1}\ne x_{2}$
and by $\bot$ otherwise. On maps $f:X\to Y$ the functor $T$ is
defined as $(Tf)(\bot)=\bot$ and $(Tf)(x_{1},x_{2})=\begin{cases}
\bot & f(x_{1})=f(x_{2})\\
(f(x_{1}),f(x_{2})) & else
\end{cases}$. Then $T$ is a functor and the projection $\pi_{\Delta}:X^{2}\to X^{2}/\Delta$
is a natural transformation. Even though $T(\emptyset)=\emptyset$,
the functor does have a constant, $\bot$.
\end{example}

The map $\delta=(T\pi_{1},T\pi_{2}):T(X\times Y)\to T(X)\times T(Y)$
is surjective: If $X=\emptyset$ or $Y=\emptyset$ this is trivial,
otherwise fix some $x\in X$ and $y\in Y$. $((x_{1},x_{2}),(y_{1},y_{2}))\in T(X)\times T(Y)$
has preimage $((x_{1},y_{1}),(x_{2},y_{2})),$ next $((x_{1},x_{2}),\bot)$
resp. $(\bot,(y_{1},y_{2}))$ have preimages $((x_{1},y)(x_{2},y))$
resp. $((x,y_{1}),(x,y_{2}))$, finally $(\bot,\bot)$ has preimage
$\bot$. Thus $T$ weakly preserves products.

To see that $T$ does not weakly preserve pullbacks of constant maps,
consider $c_{0}^{X},c_{1}^{X}:X\to\{0,1\}$ whose pullback is $\emptyset$.
But $T(c_{0}^{X})=T(c_{1}^{X})=c_{\bot}^{T(X)}$ and their pullback
is $T(X)\times T(X).$ Clearly there is no way to find a surjective
map from $T(\emptyset)=\emptyset$ to $T(X)\times T(Y)$ as would
be required by lemma \ref{lem:weak pullback}.

\section{Proof of the main theorem}

We are finally turning to the proof of theorem \ref{thm:main}, verifying
the surjectivity of $\delta=(F\pi_{1},F\pi_{2})$ when $(F,\iota,\mu)$
is a monad. Thus given $(p,q)\in F(A_{1})\times F(A_{2}),$ we are
required to find an element $t\in F(A_{1}\times A_{2})$ such that
$(F\pi_{1})(t)=p$ and $(F\pi_{2})(t)=q.$

For each $a\in A_{1}$ we define a map $\sigma_{a}:A_{2}\to A_{1}\times A_{2}$
by 
\[
\sigma_{a}(b):=(a,b),
\]
next we define $\tau:A_{1}\to F(A_{1}\times A_{2})$ by 
\[
\tau(a):=(F\sigma_{a})(q).
\]
The following picture gives an overview, where the lower squares commute
due to the fact that $\mu$ is a natural transformation, 
\[
\xyC{3pc}\xyR{3pc}\xymatrix{A_{1}\ar[dr]_{\tau}\ar@{..>}[d]_{\iota_{A_{1}}}\ar@{..>}[drr]\sp(0.7){c_{q}^{A_{1}}} & A_{1}\times A_{2}\ar@<-.15pc>[r]_{\pi_{2}}\ar[l]_{\pi_{1}} & A_{2}\ar@/^{-.3pc}/@<-.25pc>[l]_{\sigma_{a}}\\
F(A_{1}) & F(A_{1}\times A_{2})\ar[r]_{F\pi_{2}}\ar[l]^{F\pi_{1}} & F(A_{2})\\
F(F(A_{1}))\ar[u]_{\mu_{A_{1}}} & F(F(A_{1}\times A_{2}))\ar[l]^{FF\pi_{1}}\ar[r]_{FF\pi_{2}}\ar[u]^{\mu_{A_{1}\times A_{2}}} & F(F(A_{2}))\ar[u]_{\mu_{A_{2}}}
}
.
\]
and the commutativities involving the dotted arrows will be established
in the following auxiliary lemma:
\begin{lem}
\label{lem:aux}~ 

\begin{itemize}
\item $F\pi_{1}\circ\tau=\iota_{A_{1}}$
\item \textup{$F\pi_{2}\circ\tau=c_{q}^{A_{1}}$}
\end{itemize}
\end{lem}

\begin{proof}
From the definition it follows that $\pi_{1}\circ\sigma_{a}=c_{a}^{A_{2}}$
and $\pi_{2}\circ\sigma_{a}=id_{A_{2}}.$ Using these, and lemma \ref{lem:F(constant)},
we calculate:

\begin{eqnarray*}
(F\pi_{1}\circ\tau)(a) & = & (F\pi_{1})(\tau(a))\\
 & = & (F\pi_{1})((F\sigma_{a})(q))\\
 & = & ((F\pi_{1})\circ F\sigma_{a}))(q)\\
 & = & F(\pi_{1}\circ\sigma_{a})(q)\\
 & = & (Fc_{a}^{A_{2}})(q)\\
 & = & c_{\iota_{A_{1}}(a)}^{F(A_{2})}(q)\\
 & = & \iota_{A_{1}}(a)
\end{eqnarray*}
 and similarly
\begin{eqnarray*}
(F\pi_{2}\circ\tau)(a) & = & F(\pi_{2})((F\sigma_{a})(q))\\
 & = & F(\pi_{2}\circ\sigma_{a})(q)\\
 & = & F(id_{A_{2}})(q)\\
 & = & id_{F(A_{2})}(q)\\
 & = & q
\end{eqnarray*}
whence $(F\pi_{2}\circ\tau)$ is the constant map $c_{q}^{A_{1}}:A_{1}\to F(A_{2})$
.
\end{proof}
With these lemmas in place, we can finish the proof of theorem \ref{thm:main}.
We set 
\[
t:=(\mu_{A_{1}\times A_{2}}\circ F\tau)(p)
\]
and claim:
\begin{eqnarray}
(F\pi_{1})(t) & = & p\label{eq:1}\\
(F\pi_{2})(t) & = & q.\label{eq:2}
\end{eqnarray}
In order to show \ref{eq:1}, we calculate, using naturality of $\mu$,
for $i=1,2:$ 
\begin{eqnarray*}
(F\pi_{i})(t) & = & (F\pi_{i})((\mu_{A_{1}\times A_{2}}\circ F\tau)(p))\\
 & = & (F\pi_{i}\circ\mu_{A_{1}\times A_{2}}\circ F\tau)(p)\\
 & = & (\mu_{A_{i}}\circ FF\pi_{i}\circ F\tau)(p)\\
 & = & (\mu_{A_{i}}\circ F(F\pi_{i}\circ\tau))(p).
\end{eqnarray*}
Then for and for $i=1$ we continue, using lemma \ref{lem:aux} and
the first monad law: 
\begin{eqnarray*}
(\mu_{A_{1}}\circ F(F\pi_{1}\circ\tau))(p) & = & (\mu_{A_{1}}\circ F\iota_{A_{1}})(p)\\
 & = & id_{F(A_{1})}(p)\\
 & = & p,
\end{eqnarray*}
whereas for $i=2$ we obtain, using lemmas \ref{lem:aux} and \ref{lem:F(constant)}
as well as the second monad law:

\begin{eqnarray*}
(\mu_{A_{2}}\circ F(F\pi_{2}\circ\tau))(p) & = & (\mu_{A_{2}}\circ F(c_{q}^{A_{1}}))(p)\\
 & = & (\mu_{A_{2}}\circ c_{\iota_{F(A_{2})}(q)}^{F(A_{1})})(p)\\
 & = & \mu_{A_{2}}(\iota_{F(A_{2})}(q))\\
 & = & (\mu_{A_{2}}\circ\iota_{F(A_{2})})(q)\\
 & = & q.
\end{eqnarray*}

\begin{cor}
Let $\alpha=Ker\,\pi_{1}$ and $\beta=Ker\,\pi_{2},$ then 
\[
F(A\times B)/\alpha\wedge\beta\cong F(A)\times F(B).
\]
\end{cor}

\bibliographystyle{spmpsci}
\bibliography{infobib}

\end{document}